\documentclass{amsart}
\RequirePackage{amsmath}
\RequirePackage{amsthm}
\RequirePackage{amsfonts}
\RequirePackage{amssymb}
\RequirePackage{multirow}
\RequirePackage{multicol}
\RequirePackage{graphicx}
\RequirePackage{float}
\RequirePackage{enumerate} 
\RequirePackage{cancel}
\RequirePackage{appendix}
\RequirePackage{calrsfs}
\RequirePackage{cite}
\RequirePackage[figurewithin=section]{caption}
\RequirePackage{xcolor}
\RequirePackage{hyperref}
\RequirePackage{tikz}
\usetikzlibrary{babel}
\usetikzlibrary{shapes,arrows}
\tikzstyle{block} = [draw, fill=blue!20, rectangle, 
    minimum height=3em, minimum width=3em]
\tikzstyle{sum} = [draw, fill=blue!20, circle, node distance=1cm]
\tikzstyle{input} = [coordinate]
\tikzstyle{output} = [coordinate]

\newtheorem{theorem}{Theorem}[section]

\newtheorem{lemma}[theorem]{Lemma}

\theoremstyle{definition}

\newtheorem{definition}[theorem]{Definition}

\newtheorem{remark}[theorem]{Remark}
\newtheorem{example}[theorem]{Example}

\newcommand{\C}{\mathbb{C}}

\newcommand{\R}{\mathbb{R}}
\newcommand{\D}{\mathbb{D}}
\newcommand{\N}{\mathbb{N}}

\renewcommand{\H}{\mathbb{H}}

\newcommand{\abs}[1]{\left| #1 \right|}

\title{The Slope Problem in Discrete Iteration}

\date{\today}
\thanks{This research was supported in part by Ministerio de Innovaci\'on y Ciencia, Spain, project PID2022-136320NB-I00. The second author was supported by Ministerio de Universidades, Spain, through the action Ayuda del Programa de Formaci\'on de Profesorado Universitario, reference FPU21/00258.}
\author[M.D. Contreras]{Manuel D. Contreras}
\address{Departamento de Matem\'atica Aplicada II and IMUS, Escuela T\'ecnica Superior de Ingenier\'ia, Universidad de Sevilla,
	Camino de los Descubrimientos, s/n 41092, Sevilla, Spain}
\email{contreras@us.es}
\author[F. J. Cruz-Zamorano]{Francisco J. Cruz-Zamorano}
\address{Departamento de Matem\'atica Aplicada II and IMUS, Escuela T\'ecnica Superior de Ingenier\'ia, Universidad de Sevilla,
	Camino de los Descubrimientos, s/n 41092, Sevilla, Spain}
\email{fcruz4@us.es}
\author[L. Rodr\'iguez-Piazza]{Luis Rodr\'iguez-Piazza}
\address{Departmento de An\'alisis Matem\'atico and IMUS, Facultad de Matem\'aticas, Universidad
	de Sevilla, Calle Tarfia, s/n 41012 Sevilla, Spain}
\email{piazza@us.es}

\begin{document}
\maketitle
\begin{abstract}
The slope problem in holomorphic dynamics in the unit disk goes back to Wolff in 1929. However, there have been several contributions to this problem in the last decade. In this article the problem is revisited, comparing the discrete and continuous cases. Some advances are derived in the discrete parabolic case of zero hyperbolic step, showing that the set of slopes has to be a closed interval which is independent of the initial point. The continuous setting is used to show that any such interval is a possible example. In addition, the set of slopes of a family of parabolic function is discussed, leading to examples of functions with some regularity whose set of slopes is non-trivial.
\end{abstract}
\section{Introduction}
One of the goals in the field of Discrete Complex Dynamics is to study the asymptotic behaviour of the iterated self-composition $g^{n+1} = g^n \circ g$, $n \in \N$, for a given holomorphic function $g \colon \D \to \D$. Among other properties, this behaviour strongly depends on the fixed points of $g$ on $\D$. For example, if $g$ has a (necessarily unique) fixed point $\tau \in \D$ (i.e., $g(\tau) = \tau$) and it is not an automorphism of $\D$, then the Schwarz-Pick lemma assures that $g^n \to \tau$ as $n \to + \infty$ uniformly on compact subsets of $\D$. In this case, $g$ is said to be an elliptic function, and $\tau$ is known as its Denjoy-Wolff point. Concerning non-elliptic functions, Denjoy and Wolff proved the following well-known result:

\begin{theorem}[Denjoy-Wolff]
\cite[Theorem 3.2.1]{AbateNewBook}
Let $g \colon \D \to \D$ be a non-elliptic holomorphic function. Then there exists a point $\tau \in \partial\D$ such that $g^n \to \tau$, uniformly on compact sets of $\D$. Furthermore, $\tau$ is a boundary fixed point, that is, $\angle\lim_{z \to \tau}g(z) = \tau$.
\end{theorem}
Similarly as before, $\tau$ is called the Denjoy-Wolff point of $g$.

This work is devoted to non-elliptic functions, which can be further divided in two subgroups whose dynamical properties are usually different. To present them, one can show (see \cite[Corollary 2.5.5]{AbateNewBook}) that the following limit exists:
$$\lambda = \angle\lim_{z \to \tau}\dfrac{g(z)-\tau}{z-\tau} \in (0,1].$$
If $\lambda < 1$, $g$ is called hyperbolic. In the case that $\lambda = 1$, $g$ is said parabolic.

From an initial point $w_0 \in \D$, an orbit $\{w_n\} \subset \D$ is defined as
$$w_{n+1} = g(w_n) = g^{n+1}(w_0), \quad n \in \N.$$
For non-elliptic functions $g$, the Slope Problem consists on determining the directions through which the orbits approach the boundary Denjoy-Wolff point of $g$. That is, given $w_0 \in \D$, calculating the numbers $s \in [-\pi/2,\pi/2]$ such that there exists a subsequence of the orbit satisfying $\arg(1-\overline{\tau}w_{n_k}) \to s$.

In this context, it might be useful to substitute $\D$ as a domain in order to ease the calculations. To do so, if $g$ is a non-elliptic function, then let $S \colon \D \to \H$ be a Möbius transformation with $S(\tau) = \infty$, where
$\H = \{z \in \C \mid \text{Im}(z) > 0\}$
is the upper halfplane. Using $S$, $g$ can be conjugated to a holomorphic function $f \colon \H \to \H$, given by $f = S^{-1} \circ g \circ S$. This function $f$ is non-elliptic, and its Denjoy-Wolff point is $\infty$. That is, for every $z_0\in \H$, $f^n(z_0)$ converges to $\infty$ as $n$ goes to $\infty$.

Following the same conjugation, pick $w_0 \in \D$ and let $z_0 = S(w_0) \in \H$.
Define the orbit
$$z_{n+1} = f(z_n) = f^{n+1}(z_0), \quad n \in \N.$$
It is possible to check that $\arg(1-\overline{\tau}w_{n_k}) \to s$ if and only if $\pi/2-\arg(z_{n_k}) \to s$. Using this notation, the following idea can be introduced:
\begin{definition}
Let $f \colon \H \to \H$ be a holomorphic function whose Denjoy-Wolff point is $\infty$. Then, given $z_0 \in \H$, we define the set of slopes of $f$ as
$$\mathrm{Slope}[f,z_0] = \{s \in [0,\pi] : \exists \{n_k\} \subset \N \text{ with } \arg(z_{n_k}) \to s\},$$
where $\arg$ denotes the principal branch of the argument function.
\end{definition}

The properties of the set $\mathrm{Slope}[f,z_0]$ depends on the properties of $f$. For example, if the function $f$ is hyperbolic, the slope problem goes back to Wolff \cite{WolffSlope} and Valiron \cite{Valiron}. Namely, it turns out that the orbits converge non-tangentially to the Denjoy-Wolff point with a definite slope. That is, $\mathrm{Slope}[f,z]$ is a singleton in $(0,\pi)$ for all $z\in \H$ (see \cite[Theorem 4.3.4]{AbateNewBook}). Recently, Bracci and Poggi-Corradini proved that the function $\mathrm{Slope}[f,\cdot]\colon \H\to (0,\pi)$ is surjective and harmonic \cite[Property 2 (a) and Property 2 (b)]{BracciCorradiniValiron}.

When $f$ is parabolic, the behavior of the slope strongly depends on the hyperbolic step of the function (see Definition \ref{def:hypstep}). If $f$ is of positive hyperbolic step, Pommerenke proved that $\mathrm{Slope}[f,z] \subset \{0,\pi\}$ for all $z\in \H$; see \cite[Remark 1]{PommerenkeHalfPlane}. This can be slightly improved, as noted in Remark \ref{remark:PHP}: the set of slopes is always a singleton and does not depend on the starting point.

There are not many references in the literature about this problem when $f$ is of zero hyperbolic step. Indeed, the study of such scenario is the purpose of this paper. One of the first references on this is the remarkable paper of Wolff \cite[Section 6]{WolffSlope}, where he found a parabolic function $f$ such that $\mathrm{Slope}[f,z_{0}]$ contains at least two points. In the last two decades there are some new results on the set of slopes of parabolic functions that can be embedded into a continuous semigroup (see \cite{BetsakosSlope,CDMGSlope,KelgiannisSlope}) but, as far as we know, there are no advances in the discrete setting. In Section \ref{sec:Parabolic}, we obtain that the set of slopes of a parabolic function of zero hyperbolic step is always an interval which does not depend on the initial point of the orbit (see Theorem \ref{teo:0HS}). We also show that for any close interval $[a,b]\subset [0,\pi]$ there exists a parabolic function $f$ of zero hyperbolic step such that $\mathrm{Slope}[f,z]=[a,b]$ (see Theorem \ref{teo:0HS2}).

The Slope Problem can also be introduced in the continuous setting of Complex Dynamics. Indeed, the results that we have just introduced have an exact analogue in this other setting; see \cite[Sections 17.4-6]{Contreras} for reference. There, the orbits are continuous curves defined through non-elliptic semigroups (see Definition \ref{def:semigroup}), and the directions through which orbits converge to the boundary are also studied. Even if the results are conceptually similar, the fact that the orbits can either be connected or discrete sets completely changes the problem from a mathematical point of view.

To develop some new ideas, we also introduce a family of parabolic functions that has been previously considered in the literature \cite{Aaronson,HS}. This family is defined in terms of a real parameter and a positive finite measure. In Section \ref{sec:integrable} we characterize the set of slopes of these functions depending on some integrability requirements. Using these ideas we are able to provide explicit examples of parabolic functions $f$ of zero hyperbolic step with some regularity at the Denjoy-Wolff point for which $\mathrm{Slope}[f,z_0] = [0,\pi]$ or $\mathrm{Slope}[f,z_0] = [0,\pi/2]$. This highly contrasts with some known results on semigroups, where regularity is used to assure that the orbits converge to the boundary with a definite slope; see \cite[Theorem 21]{Regularity} or \cite[Proposition 17.5.5]{Contreras}.

\section{Parabolic functions}
\label{sec:Parabolic}
From now on, we will work with functions $f \colon \H \to \H$ that are parabolic and whose Denjoy-Wolff point is $\infty$. These functions can be split into two different classes. To introduce them, let $\rho$ be the pseudo-hyperbolic distance in $\H$ given by
$$\rho(z,w) = \abs{\dfrac{z-w}{z-\overline{w}}}, \quad z,w \in \H.$$
With the aid of Schwarz-Pick Lemma, it is clear that the sequence $\rho(z_{n+1},z_n) = \rho(f(z_n),f(z_{n-1}))$ is non-increasing for any initial point $z_0 \in \D$, and therefore it must converge. In fact, if its limit is zero for some initial point, then it is also zero for every initial point, as stated in \cite[Corollary 4.6.9.(i)]{AbateNewBook}. With this idea, the following property can be defined:
\begin{definition} \label{def:hypstep}
Let $f \colon \H \to \H$ be a parabolic function. We say that $f$ is of zero hyperbolic step if $\rho (z_{n+1},z_n) \to 0$ as $n \to +\infty$ for some (equivalently, for every) $z_0 \in \H$. Otherwise, we say $f$ is of positive hyperbolic step.
\end{definition}
Determining general conditions to distinguish the hyperbolic step of parabolic functions is a subtle task. For instance, the following is a well-known result about that:
\begin{theorem}
\cite[Theorem 1]{PommerenkeHalfPlane}
\label{teo:Pomm}
Let $f \colon \H \to \H$ be a parabolic function with Denjoy-Wolff point at $\infty$. The following limit converges:
$$b = \lim_{n \to + \infty}\dfrac{x_{n+1}-x_n}{y_n} \in \R,$$
where $z_n = x_n + iy_n = f^n(z_0)$, $n \in \N,$ is the orbit of some initial point $z_0 \in \H$. Furthermore, $z_{n+1}/z_n \to 1$ and $y_{n+1}/y_n \to 1$ as $n \to + \infty$. Additionally, $b = 0$ if and only if $f$ is of zero hyperbolic step.
\end{theorem}
\begin{remark}
\label{remark:PHP}
Using the ideas in the last result, Pommerenke \cite[Remark 1]{PommerenkeHalfPlane} proved that $\mathrm{Slope}[f,z_0] \subset \{0,\pi\}$ if $f$ is a parabolic function of positive hyperbolic step. However, note that if $z,w \in \H$ are such that $\mathrm{Slope}[f,z] = \{0\}$ and $\mathrm{Slope}[f,w] = \{\pi\}$, then $\rho(f^n(z),f^n(w)) \to 1$, which contradicts the Schwarz-Pick Lemma. Therefore, one of the following must apply: $\mathrm{Slope}[f,z] = \{0\}$ for every $z \in \H$ or $\mathrm{Slope}[f,z] = \{\pi\}$ for every $z \in \H$. 

Indeed, for any given $z_0 \in \H$, Theorem \ref{teo:Pomm} implies that 
\begin{align*}
\lim_{n \to + \infty} \dfrac{x_n}{y_n} & = \lim_{n \to + \infty} \dfrac{x_{n+1}-x_n}{y_{n+1}-y_n} =  \lim_{n \to + \infty} \dfrac{x_{n+1}-x_n}{y_n}\dfrac{y_n}{y_{n+1}-y_n} = \alpha \cdot (+\infty),
\end{align*}
where we have used that $y_{n+1}/y_n \to 1$, as shown in Theorem \ref{teo:Pomm}. In that case, it is obvious that $\mathrm{Slope}[f,z_0] = \{0\}$ if and only if $\alpha > 0$. Therefore, the sign of the limit $\alpha = \alpha(z_0)$ is independent of the initial point.
\end{remark}

In contrast to the case of hyperbolic or parabolic functions with positive hyperbolic step, there are not many references concerning the zero hyperbolic step situation. To develop some advances on this topic, we are lead by the analogue results in continuous iteration. In seek of completeness, let us start by stating a lemma on sequences:

\begin{lemma}
\label{lemma:limitpoints}
Let $\{x_n\}$ be a bounded sequence with $x_{n+1}-x_n \to 0$ as $n \to + \infty$. The set of limit points of the sequence, that is, $L = \{s \in \R : \exists\{n_k\}\subset\N \text{ with } x_{n_k} \to s\}$ is a closed and connected set.
\begin{proof}
The set of limit points of a sequence is always closed. Therefore, it only remains to show that $L$ is connected. To do so, let $a,b \in L$ with $a \leq b$ and choose any $a \leq x \leq b$. We can show that $x \in L$ as follows:

Fix $\epsilon_1 = 1$. Let $N_1$ be such that $\abs{x_{n+1}-x_n} \leq 1$ for all $n \geq N_1$. By definition, there exists $n_1,m_1 \geq N_1$ such that $\abs{x_{n_1}-a} \leq 1$ and $\abs{x_{m_1}-b} \leq 1$. Without loss of generality, suppose that $n_1 \leq m_1$. Then, there exists $n_1 \leq l_1 \leq m_1$ such that $\abs{x_{l_1}-x} \leq 1$. 

Similarly, fix $\epsilon_k = 1/k$ for $k \geq 2$. Let $N_k \geq l_{k-1}$ be such that $\abs{x_{n+1}-x_n} \leq \epsilon_k$ for all $n \geq N_k$. By definition, there exists $n_k,m_k \geq N_k$ such that $\abs{x_{n_k}-a} \leq \epsilon_k$ and $\abs{x_{m_k}-b} \leq \epsilon_k$. Without loss of generality, suppose that $n_k \leq m_k$. Then, there exists $n_k \leq l_k \leq m_k$ such that $\abs{x_{l_k}-x} \leq \epsilon_k$. That is, $x_{l_k} \to x$, and so $x \in L$.
\end{proof}
\end{lemma}

As in the continuous setting (see \cite[p. 501]{Contreras}), we can prove that the set of slopes is a (non necessarily trivial) closed and connected set, regardless of the orbit being discrete:

\begin{theorem}
\label{teo:0HS}
Let $f \colon \H \to \H$ be a parabolic function of zero hyperbolic step with Denjoy-Wolff point at $\infty$. Then, $\mathrm{Slope}[f,z] = \mathrm{Slope}[f,w]$ for all $z,w \in \H$. Furthermore, there exists $0 \leq a \leq b \leq \pi$ such that $\mathrm{Slope}[f,z] = [a,b]$ for any $z \in \H$.
\end{theorem}
\begin{proof}
\textit{Step 1.} $\mathrm{Slope}[f,z]$ does not depend on $z$.

The idea for this proof comes from \cite[p. 51, Lemma 1.8.6]{Contreras}. Given $z_0,w_0 \in \H$, we will show that $\mathrm{Slope}[f,z_0] \subset \mathrm{Slope}[f,w_0]$. By symmetry, that is enough to conclude that the set of slopes does not depend on the starting point of the orbit.

To do so, let $s \in \mathrm{Slope}[f,z_0]$. In that case, there exists $\{n_k\} \subset \N$ such that $\arg(z_{n_k})\to s$. Then, one can write $z_{n_k} = r_{n_k}e^{i\alpha_{n_k}}$, where $r_{n_k} \to + \infty$ and $\alpha_{n_k} \to s$ as $n_k \to +\infty$.

Now, pick a limit point $s'$ of the set $\{\arg(w_{n_k})\}$ and a subsequence (which we will denote as $n_k$ as well, for brevity) such that $\arg(w_{n_k}) \to s'$. We will show that $s = s'$. Again, write $w_{n_k} = t_{n_k}e^{i\beta_{n_k}}$, where $t_{n_k} \to + \infty$ and $\beta_{n_k} \to s'$ as $n_k \to + \infty$. As $f$ is of zero hyperbolic step, by \cite[Corollary 4.6.9]{AbateNewBook}, we have
$$\rho(z_{n_k},w_{n_k}) = \rho(f^{n_k}(z_0),f^{n_k}(w_0)) \to 0, \quad \text{as } k \to +\infty.$$
But
\begin{align*}
\rho(z_{n_k},w_{n_k}) & = \abs{\dfrac{r_{n_k}e^{i\alpha_{n_k}}-t_{n_k}e^{i\beta_{n_k}}}{r_{n_k}e^{i\alpha_{n_k}}-t_{n_k}e^{-i\beta_{n_k}}}}  = \abs{\dfrac{\lambda_{n_k}e^{i\alpha_{n_k}}-(1-\lambda_{n_k})e^{i\beta_{n_k}}}{\lambda_{n_k}e^{i\alpha_{n_k}}-(1-\lambda_{n_k})e^{-i\beta_{n_k}}}},
\end{align*}
where
$$\lambda_{n_k} = \dfrac{r_{n_k}}{r_{n_k}+t_{n_k}}.$$
Suppose that $s \neq s'$. Note that $\lambda_{n_k} \in (0,1)$. In particular, choose a subsequence (again, also denoted as $n_k$) such that $\lambda_{n_k} \to \lambda \in [0,1]$. In that case, taking limits, one has
$$\dfrac{\lambda e^{is}-(1-\lambda)e^{is'}}{\lambda e^{is}-(1-\lambda)e^{-is'}} = 0,$$
where $\lambda e^{is}-(1-\lambda)e^{-is'} \neq 0$ because $s\neq s'$.
Thus $\lambda e^{is}-(1-\lambda)e^{is'} = 0$, from which it follows that $\lambda = 1/2$ and $e^{i(s-s')} = 1$. Since $s-s' \in [-\pi,\pi]$, we get $s = s'$, which contradicts the assumption. \\

\textit{Step 2.} $\mathrm{Slope}[f,z]$ is a closed interval.

Remember that $\mathrm{Slope}[f,z]$ is the set of limit points of $\{\arg(z_n) \mid n \in \N\}$. But, due to Theorem \ref{teo:Pomm}, $\arg(z_{n+1})-\arg(z_n) \to 0$ as $n \to + \infty$.
Then, Lemma \ref{lemma:limitpoints} assures that $\mathrm{Slope}[f,z]$ is a closed interval.
\end{proof}

As said before, the slope problem has also been of interest in the theory of continuous iteration, that is, the theory of semigroups. We recall:
\begin{definition}
\label{def:semigroup}
A (continuous) semigroup (of holomorphic self-maps of the upper halfplane) is a family of holomorphic functions $\{\phi_t \colon \H \to \H \mid t \geq 0\}$ such that
\begin{enumerate}[\hspace{0.5cm}(a)]
\item $\phi_0$ is the identity map in $\H$,
\item $\phi_{s+t} = \phi_s \circ \phi_t,$ for $s,t \geq 0$,
\item the map $t \in [0,+\infty) \mapsto \phi_t$ is continuous, with respect to the Euclidean topology and the topology given by the uniform convergence on compact sets of $\H$.
\end{enumerate}
\end{definition}

A semigroup $\{\phi_{t}\}$ is parabolic (resp. of zero or positive hyperbolic step) if there exists $t_{0}>0$ such that $\phi_{t_{0}}$ is parabolic (resp. of zero hyperbolic step). It can be checked that in such a case all the functions $\phi_{t}$, with $t>0$, are parabolic (resp, of zero or positive hyperbolic step); see \cite[Remark 8.3.4, Definition 9.3.4]{Contreras}.

One can also consider the set of slopes in the continuous setting: given a non-elliptic semigroup  $\{\phi_{t}\}$, the set of slopes of the semigroup at $z\in \H$ is defined as
$$\mathrm{Slope}[\{\phi_t\},z] = \{s \in [0,\pi] : \exists \{t_k\} \subset [0,+\infty),  \text{ with }t_{k}\to +\infty, \, \arg(\phi_{t_{k}}(z)) \to s\}.$$

In this context, the following is a well-known and celebrated result:
\begin{theorem}
\label{teo:semigroup}
\cite[Theorem 1]{KelgiannisSlope}
For any $0 \leq a \leq b \leq \pi$ there exists a parabolic semigroup $\{\phi_t\}$  of zero hyperbolic step whose Denjoy-Wolff point is $\infty$ such that $\mathrm{Slope}[\{\phi_t\},z] = [a,b]$ for all $z \in \H$.
\end{theorem}
\begin{remark}
The actual statement in \cite[Theorem 1]{KelgiannisSlope} is different, as they work with the set $\pi/2-\mathrm{Slope}[\{\phi_t\},z]$.
\end{remark}
The theory of semigroups and the theory of discrete iteration of holomorphic self-maps are strongly linked. Indeed, one can use the ideas for semigroups in order to give the following result:
\begin{theorem} \label{teo:0HS2}
Given $0 \leq a \leq b \leq \pi$, there exists a parabolic function $f \colon \H \to \H$ of zero hyperbolic step with Denjoy-Wolff point at $\infty$ such that $\mathrm{Slope}[f,z] = [a,b]$ for all $z \in \H$.
\begin{proof}
Using Theorem \ref{teo:semigroup}, fix a parabolic semigroup $\{\phi_t\}$ of zero hyperbolic step whose Denjoy-Wolff point is $\infty$ such that $\mathrm{Slope}[\{\phi_t\},z] = [a,b]$ for all $z \in \H$. Set $f = \phi_1$, and notice that $z_n = f^n(z_0) = \phi_n(z_0)$. We will prove that $\mathrm{Slope}[f,z] = [a,b]$. Clearly, $\mathrm{Slope}[f,z] \subset \mathrm{Slope}[\{\phi_t\},z] = [a,b]$. Then, it is enough to show that $\mathrm{Slope}[f,z] \supset \mathrm{Slope}[\{\phi_t\},z] = [a,b]$. To do so, let $s \in \mathrm{Slope}[\{\phi_t\},z]$. Thus, let $t_k \geq 0$ be an increasing sequence such that $t_k \to + \infty$ and $\arg(\phi_{t_k}(z)) \to s$. We may assume that $t_1 \geq 1$. For each $t_k$, consider $n_k \in \N$ given by $n_k \leq t_k < n_k+1$. We claim that $\rho(\phi_{t_k}(z),\phi_{n_k}(z)) \to 0$ as $k \to +\infty$, and then one can argue as in the proof of Theorem \ref{teo:0HS} to see that there exists a subsequence (again, denoted in the same way) such that $\arg(z_{n_k}) \to s$, from which we have that $s \in \mathrm{Slope}[f,z].$

To prove the claim, let $K \subset \H$ be the compact set given by $K = \{\phi_s(z) : s \in [0,1]\}$, and notice that $\phi_{n_k}(z),\phi_{t_k}(z) \in \phi_{n_k}(K)$ for all $k \in \N$. Thus, it is enough to prove that
$$\sup_{w \in K}\rho(\phi_n(z),\phi_n(w)) \to 0, \quad n \to +\infty.$$
To do this, up to a standard argument of compactness, it is enough to prove that
$$\rho(\phi_n(z),\phi_n(w)) \to 0, \quad n \to + \infty,$$
for all $w \in K$. But this is a consequence of an idea due to Pommerenke \cite[Theorem 1]{PommerenkeHalfPlane} (which can be found in \cite[Theorem 3.1]{Doering}; see also \cite[Corollary 4.6.9.(iv)]{AbateNewBook} for a proof using models).
\end{proof}
\end{theorem}

\section{Examples with definite slope}
\label{sec:integrable}
In this section, the set $\mathrm{Slope}[f,z]$ will be characterized for a family of parabolic functions that has been previously considered in the literature. To introduce it, let us start by linking \cite[Theorem 6.2.1]{Aaronson} together with \cite[Chapter 5, Lemma 2]{Doering} to obtain the following result:
\begin{theorem}
\label{teo:expr}
Every parabolic function $f \colon \H \to \H$ whose Denjoy-Wolff point is $\infty$ can be uniquely written as
$$f(z) = z + \beta + \int_{\R}\dfrac{1+tz}{t-z}d\mu(t), \quad z \in \H,$$
where $\beta \in \R$ and $\mu$ is a positive finite measure on $\R$.
\end{theorem}
In \cite[Theorem 6.4.1]{Aaronson}, Aaronson studied some properties of $f$ under the assumption that $\mu$ is of compact support. For example, if $\beta = \int_{\R}td\mu(t)$, it is shown that $x_n$ is bounded, and so $\text{Slope}[f,z] = \{\pi/2\}$ for all $z \in \H$. If $\beta \neq \int_{\R}td\mu(t)$, then $y_n$ is bounded. Indeed, if $\beta > \int_{\R}td\mu(t)$, then $x_n \to + \infty$ and $\text{Slope}[f,z] = \{0\}$ for all $z \in \H$. Similarly, if $\beta < \int_{\R}td\mu(t)$, then $x_n \to - \infty$ and $\text{Slope}[f,z] = \{\pi\}$ for all $z \in \H$.

In this section, we generalize the former ideas of Aaronson for a larger family of functions. To introduce them, suppose that
$$\int_{\R}\abs{t}d\mu(t) < + \infty.$$
In that case, $f$ can be rewritten as
$$f(z) = z + \beta - \int_{\R}td\mu(t) + p(z), \quad p(z) = \int_{\R}\dfrac{1+t^2}{t-z}d\mu(t), \quad z \in \H.$$
This family of functions has also been considered in \cite{HS}, where its hyperbolic step is characterized in terms of $\mu$ and $\beta$.

Let us begin by stating the following consequence of Lebesgue's Dominated Convergence Theorem: 
\begin{lemma}\cite[Lemma 3.3]{HS}
\label{lemma:anglelimp}
Assume that $\displaystyle\int_{\R}\abs{t}d\mu(t) < \infty$. Then, $\displaystyle\angle\lim_{z \to \infty}p(z) = 0$.
\end{lemma}
Using the former description, we can now show a result on the set of slopes of convergence:
\begin{theorem}
\label{teo:tildebeta0t2}
Assume that $\displaystyle\int_{\R}t^2d\mu(t) < \infty$ and $\beta = \displaystyle\int_{\R}td\mu(t)$. Then $\mathrm{Slope}[f,z] = \{\pi/2\}$ for all $z \in \H$.
\begin{proof}
First of all, notice that using Lebesgue's Dominated Convergence Theorem, one can see that
$$\angle\lim_{z \to \infty}zp(z) = -\int_{\R}(1+t^2)d\mu(t) < 0,$$
which implies that $\angle\lim_{z \to \infty}\text{arg}(zp(z)) = \pi$. Then, we rewrite
$$f(z) = z - \dfrac{q(z)}{z}, \quad z \in \H,$$
and
\begin{equation}\label{Eq:q}
\angle\lim_{z \to \infty}q(z) > 0.
\end{equation}

With this idea in mind, fix $0 < \alpha < \pi/2$, and define $S = \{z \in \H : \abs{\arg(z)-\pi/2}<\alpha\}$. We will show that there exists $L = L(\alpha)$ such that $f(S^*) \subset S^*$, where $S^* = S \cap \{z \in \H : \text{Im}(z) > L\}.$ To do so, we split $S$ in three parts as follows: define $S_0 = \{z \in \H : \abs{\arg(z)-\pi/2}\leq\alpha/2\} \subset S$, and also $S_- = (S \setminus S_0) \cap \{z \in \H : \text{Re}(z) < 0\}$, $S_+ = S \setminus(S_0 \cup S_-)$.

Note that, due to Lemma \ref{lemma:anglelimp}, we can choose $L_1 > 0$ such that
$$\abs{\dfrac{q(z)}{z}} < 1, \quad z \in S_0, \, \text{Im}(z) > L_1.$$
In that case, it is possible to find $L_1^* \geq L_1$ such that, for all $L > L_1^*$, $f$ maps $S_0^* = S_0 \cap \{z \in \H : \mathrm{Im}(z) > L\}$ into $S^*$.

Similarly, by \eqref{Eq:q}, it is possible to find $L_2 > 0$ such that
$$\abs{\arg(q(z))} \leq \dfrac{\alpha}{4}, \quad z \in S_-, \, \text{Im}(z) > L_2.$$
In that case,
$$\dfrac{\pi}{2}-\dfrac{5\alpha}{4} \leq \arg\left(-\dfrac{q(z)}{z}\right) \leq \dfrac{\pi}{2}-\dfrac{\alpha}{4}, \quad z \in S_-, \, \text{Im}(z) > L_2.$$
In particular, this means that $\arg(f(z)) > \arg(z)$ if $z \in S_-$ with $\text{Im}(z) > L_2$. We may also suppose that
$$\abs{-\dfrac{q(z)}{z}} < 1, \quad z \in S_-, \, \text{Im}(z) > L_2.$$
In that case, it is possible to find $L_2^* \geq L_2$ such that, or all $L > L_2^*$, $f$ maps $S_-^* = S_- \cap \{z \in \H : \mathrm{Im}(z) > L\}$ into $S^*$. Similarly, it is also possible to find $L_3^* > 0$ such that, for all $L > L_3^*$, $f$ maps $S_+^* = S_+ \cap \{z \in \H : \text{Im}(z) > L\}$ into $S^*$.

Summing up, if $L > \max\{L_1^*,L_2^*,L_3^*\}$, then $S^* = S_-^* \cup S_0^* \cup S_+^*$, with $f(S_-^*) \subset S^*$, $f(S_0^*) \subset S^*$, and $f(S_+^*) \subset S^*$. Then $f(S^*) \subset S^*$, as we wanted to prove.

Using the former ideas, one can see that, for any $0 < \alpha < \pi/2$, it is possible to find $y_0 = y_0(\alpha) > 0$ such that $\mathrm{Slope}[f,iy_0] \subset [\pi/2-\alpha,\pi/2+\alpha]$. Using Theorem \ref{teo:0HS}, we conclude that $\mathrm{Slope}[f,z] \subset \{\pi/2\}$ for all $z \in \H$.
\end{proof}
\end{theorem}
The case where $\beta \neq \displaystyle\int_{\R}td\mu(t)$ can be characterized under more general hypotheses:
\begin{theorem}
\label{teo:tildebetal1}
Assume that $\displaystyle\int_{\R}\abs{t}d\mu(t) < \infty$. If $\beta > \displaystyle\int_{\R}td\mu(t)$, then $\mathrm{Slope}[f,z] = \{0\}$ for all $z \in \H$. If $\beta < \displaystyle\int_{\R}td\mu(t)$, then $\mathrm{Slope}[f,z] = \{\pi\}$ for all $z \in \H$.
\begin{proof}
Let us define
$$\tilde{\beta} = \beta - \int_{\R}td\mu(t).$$
Assume $\tilde{\beta} > 0$ (the case $\tilde{\beta} < 0$ follows using similar ideas). We claim that
$$\lim_{\substack{z \to \infty \\ \text{Im}(z) > L}}\dfrac{p(z)}{z} = 0,$$
for any $L > 0$. To prove the claim, we use that if $t \in \R$ and $z = x+iy \in \H$, then
$$\abs{\dfrac{t}{t-z}} \leq \dfrac{\abs{z}}{y}.$$
Notice that this is equivalent to
$$((t-x)^2+y^2)(x^2+y^2)-y^2t^2 = (xt-x^2-y^2)^2 \geq 0.$$
Thus, if $\abs{t} \geq 1$,
$$\dfrac{1}{\abs{z}}\dfrac{1+t^2}{\abs{t-z}} \leq \dfrac{1}{\abs{z}}\dfrac{2t^2}{\abs{t-z}} \leq \dfrac{2}{y}\abs{t} \leq \dfrac{2}{L}\abs{t}, \quad x \in \R, \, y > L.$$
Therefore, the claim follows from Lebesgue's Dominated Convergence Theorem.

With this in mind, fix $0 < \alpha < \pi/4$ and define $S = \{z \in \H : \arg(z) < 2\alpha\}$. Split $S$ in two domains as follows: define $S_u = \{z \in \H : \alpha \leq \arg(z) < 2\alpha\}$ and $S_l = S \setminus S_u$. Note that $S_u$ is a non-tangential region at $\infty$. Then, by Lemma \ref{lemma:anglelimp}, we have
$$\lim_{\substack{z \to \infty \\ z \in S_u}}\arg(\tilde{\beta}+p(z)) = 0.$$
In that case, it is possible to find $L_1 > 0$ such that $\arg(f(z)-z) = \arg(\tilde{\beta}+p(z)) < \alpha$, if $z \in S_u$ with $\text{Im}(z) > L_1$. Thus, $\arg(f(z)) < 2\alpha$, because $S$ is a cone.

Similarly, one has
$$\lim_{\substack{z \to \infty \\ z \in S_l, \, \text{Im}(z) > L}}\left(\dfrac{\tilde{\beta}}{z}+\dfrac{p(z)}{z}\right) = 0,$$
for any $L > 0$. Then, it is possible to find $L_2 > 0$ such that
$$\arg\left(1+\dfrac{\tilde{\beta}}{z}+\dfrac{p(z)}{z}\right) < \alpha, \quad z \in S_l, \, \text{Im}(z) > L_2.$$
Thus,
$$\arg(f(z)) = \arg(z) + \arg\left(1+\dfrac{\tilde{\beta}}{z}+\dfrac{p(z)}{z}\right) < 2\alpha, \quad z \in S_l, \, \text{Im}(z) > L_2.$$

We may suppose that $L_1 = L_2$. Then, define $S^* = S \cap \{z \in \H : \text{Im}(z) > L_1\}$, $S_u^* = S_u \cap \{z \in \H : \text{Im}(z) > L_1\}$ and $S_l^* = S_l  \cap \{z \in \H : \text{Im}(z) > L_1\}$. It is clear that $S^* = S_u^* \cup S_l^*$, and we just justified that $f(S_u^*) \subset S^*$ and $f(S_l^*) \subset S^*$. So, $f(S^*) \subset S^*$.

In that case, for any $0 < \alpha < \pi/4$, one could choose $z_0 \in S^* = S^*(\alpha)$ and then it is clear that $\mathrm{Slope}[f,z_0] \subset [0,2\alpha]$. Then, by Theorem \ref{teo:0HS}, we conclude that $\mathrm{Slope}[f,z] = \{0\}$ for all $z \in \H$.
\end{proof}
\end{theorem}
In contrast with Theorem \ref{teo:tildebetal1}, Theorem \ref{teo:tildebeta0t2} does not hold under the more general hypothesis given by
$$\int_{\R}\abs{t}d\mu(t) < + \infty.$$
Indeed, for each $s \in (0,\pi)$, we now present a function $f \colon \H \to \H$ that satisfies this hypothesis but $\text{Slope}[f,z]=\{s\}$. To do this, let $m$ denotes the Lebesgue measure on $\R$, and let us proceed with the following examples:
\begin{example}
\label{exp:pi2}
Let $\mu$ be the positive finite measure on $\R$ given by
$$\dfrac{d\mu}{dm}(t) = \dfrac{1}{(1+t^2)t}\chi_{(1,+\infty)}(t), \quad t \in \R.$$
Note that
$$\int_{\R}\abs{t}d\mu(t) < +\infty, \quad \int_{\R}t^2d\mu(t) = + \infty.$$

With the former notation, let $f \colon \H \to \H$ be given by
$$f(z) = z + p(z), \quad z \in \H,$$
where
$$p(z) = \int_{\R}\dfrac{1+t^2}{t-z}d\mu(t) = \int_{(1,+\infty)}\dfrac{dt}{(t-z)t} = -\dfrac{\log(1-z)}{z}.$$

One can check that $\lim_{z \to \infty}\arg(zp(z)) = \pi$. Then, following the same ideas as in the proof of Theorem \ref{teo:tildebeta0t2}, it is possible to see that $\mathrm{Slope}[f,z] = {\pi/2}$ for all $z \in \H$.
\end{example}
\begin{remark}
Following the notation on Theorem \ref{teo:expr}, note that $\mathrm{Slope}[f,z] = {\pi/2}$ if $\mu$ is a symmetric measure and $\beta = 0$; see the discussion preceding the Proposition 4.1 in \cite{HS}. However, Example \ref{exp:pi2} shows that this is not necessary, even in the case where
$$\int_{\R}\abs{t}d\mu(t) < +\infty, \quad \int_{\R}t^2d\mu(t) = + \infty.$$
\end{remark}
\begin{example}
For $0 < \alpha < 1$, let $\mu_{\alpha}$ be the positive finite measure on $\R$ given by
$$\dfrac{d\mu_{\alpha}}{dm}(t) = \dfrac{1}{(1+t^2)t^{\alpha}}\chi_{(0,+\infty)}(t), \quad t \in \R.$$

With the former notation, let $f_{\alpha} \colon \H \to \H$ be given by
$$f_{\alpha}(z) = z + p_{\alpha}(z), \quad z \in \H,$$
where
$$p_{\alpha}(z) = \int_{(0,+\infty)}\dfrac{dt}{(t-z)t^\alpha}.$$
This integral can be explicitly calculated by using Cauchy's Residue Theorem. In order to do so, we use some curves given in the following figure:
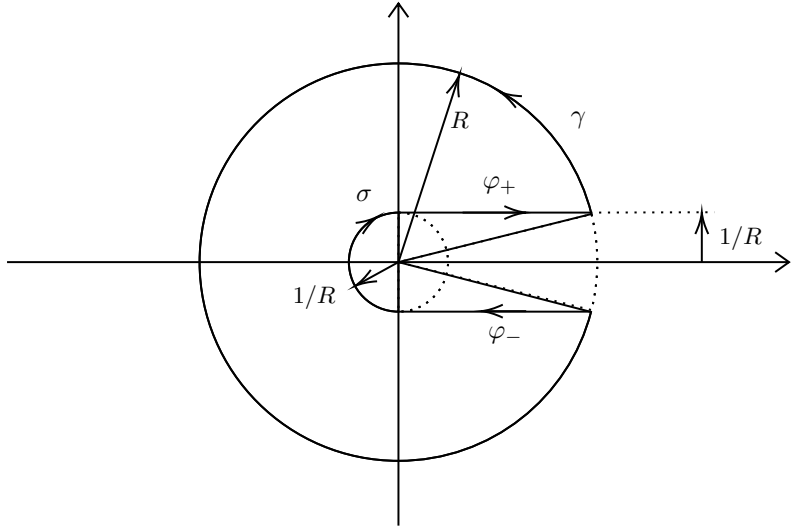
\begin{figure}[H]
\centering
\tikzset{every picture/.style={line width=0.75pt}}
\begin{tikzpicture}[x=0.75pt,y=0.75pt,yscale=-1,xscale=1]
\draw  (10,139.6) -- (404.4,139.6)(207.4,9) -- (207.4,272.6) (397.4,134.6) -- (404.4,139.6) -- (397.4,144.6) (202.4,16) -- (207.4,9) -- (212.4,16)  ;
\draw  [draw opacity=0] (304.52,164.74) .. controls (293.37,207.97) and (254.11,239.9) .. (207.4,239.9) .. controls (152.01,239.9) and (107.1,194.99) .. (107.1,139.6) .. controls (107.1,84.21) and (152.01,39.3) .. (207.4,39.3) .. controls (254.42,39.3) and (293.88,71.65) .. (304.74,115.31) -- (207.4,139.6) -- cycle ; \draw   (304.52,164.74) .. controls (293.37,207.97) and (254.11,239.9) .. (207.4,239.9) .. controls (152.01,239.9) and (107.1,194.99) .. (107.1,139.6) .. controls (107.1,84.21) and (152.01,39.3) .. (207.4,39.3) .. controls (254.42,39.3) and (293.88,71.65) .. (304.74,115.31) ; 
\draw  [draw opacity=0][dash pattern={on 0.84pt off 2.51pt}] (304.76,115.4) .. controls (306.68,123.15) and (307.7,131.26) .. (307.7,139.6) .. controls (307.7,147.94) and (306.68,156.04) .. (304.76,163.79) -- (207.4,139.6) -- cycle ; \draw  [dash pattern={on 0.84pt off 2.51pt}] (304.76,115.4) .. controls (306.68,123.15) and (307.7,131.26) .. (307.7,139.6) .. controls (307.7,147.94) and (306.68,156.04) .. (304.76,163.79) ;  
\draw  [draw opacity=0] (207.4,164.6) .. controls (207.4,164.6) and (207.4,164.6) .. (207.4,164.6) .. controls (193.59,164.6) and (182.4,153.41) .. (182.4,139.6) .. controls (182.4,125.79) and (193.59,114.6) .. (207.4,114.6) .. controls (207.46,114.6) and (207.53,114.6) .. (207.59,114.6) -- (207.4,139.6) -- cycle ; \draw   (207.4,164.6) .. controls (207.4,164.6) and (207.4,164.6) .. (207.4,164.6) .. controls (193.59,164.6) and (182.4,153.41) .. (182.4,139.6) .. controls (182.4,125.79) and (193.59,114.6) .. (207.4,114.6) .. controls (207.46,114.6) and (207.53,114.6) .. (207.59,114.6) ;  
\draw  [draw opacity=0][dash pattern={on 0.84pt off 2.51pt}] (207.4,114.6) .. controls (221.21,114.6) and (232.4,125.79) .. (232.4,139.6) .. controls (232.4,153.41) and (221.21,164.6) .. (207.4,164.6) -- (207.4,139.6) -- cycle ; \draw  [dash pattern={on 0.84pt off 2.51pt}] (207.4,114.6) .. controls (221.21,114.6) and (232.4,125.79) .. (232.4,139.6) .. controls (232.4,153.41) and (221.21,164.6) .. (207.4,164.6) ;  
\draw    (207.4,114.6) -- (304.1,114.6) ;
\draw    (207.4,164.6) -- (304.6,164.6) ; 
\draw    (207.4,139.6) -- (237.52,46.9) ;
\draw [shift={(238.13,45)}, rotate = 108] [color={rgb, 255:red, 0; green, 0; blue, 0 }  ][line width=0.75]    (10.93,-3.29) .. controls (6.95,-1.4) and (3.31,-0.3) .. (0,0) .. controls (3.31,0.3) and (6.95,1.4) .. (10.93,3.29)   ;
\draw    (207.4,139.6) -- (187.22,150.7) ;
\draw [shift={(185.47,151.67)}, rotate = 331.18] [color={rgb, 255:red, 0; green, 0; blue, 0 }  ][line width=0.75]    (10.93,-3.29) .. controls (6.95,-1.4) and (3.31,-0.3) .. (0,0) .. controls (3.31,0.3) and (6.95,1.4) .. (10.93,3.29)   ;
\draw  [dash pattern={on 0.84pt off 2.51pt}]  (304.1,114.6) -- (366.71,114.42) ;
\draw    (360.4,139.93) -- (360.4,116.47) ;
\draw [shift={(360.4,114.47)}, rotate = 90] [color={rgb, 255:red, 0; green, 0; blue, 0 }  ][line width=0.75]    (10.93,-3.29) .. controls (6.95,-1.4) and (3.31,-0.3) .. (0,0) .. controls (3.31,0.3) and (6.95,1.4) .. (10.93,3.29)   ;
\draw    (262.27,55.67) -- (260.24,54.6) ;
\draw [shift={(258.47,53.67)}, rotate = 27.76] [color={rgb, 255:red, 0; green, 0; blue, 0 }  ][line width=0.75]    (10.93,-3.29) .. controls (6.95,-1.4) and (3.31,-0.3) .. (0,0) .. controls (3.31,0.3) and (6.95,1.4) .. (10.93,3.29)   ;
\draw    (192.13,119.67) -- (195.12,117.75) ;
\draw [shift={(196.8,116.67)}, rotate = 147.26] [color={rgb, 255:red, 0; green, 0; blue, 0 }  ][line width=0.75]    (10.93,-3.29) .. controls (6.95,-1.4) and (3.31,-0.3) .. (0,0) .. controls (3.31,0.3) and (6.95,1.4) .. (10.93,3.29)   ;
\draw    (239.75,114.6) -- (269.75,114.6) ;
\draw [shift={(271.75,114.6)}, rotate = 180] [color={rgb, 255:red, 0; green, 0; blue, 0 }  ][line width=0.75]    (10.93,-3.29) .. controls (6.95,-1.4) and (3.31,-0.3) .. (0,0) .. controls (3.31,0.3) and (6.95,1.4) .. (10.93,3.29)   ;
\draw    (271.83,164.35) -- (251.1,164.35) ;
\draw [shift={(249.1,164.35)}, rotate = 360] [color={rgb, 255:red, 0; green, 0; blue, 0 }  ][line width=0.75]    (10.93,-3.29) .. controls (6.95,-1.4) and (3.31,-0.3) .. (0,0) .. controls (3.31,0.3) and (6.95,1.4) .. (10.93,3.29)   ;

\draw (231.93,62.4) node [anchor=north west][inner sep=0.75pt]  [font=\small]  {$R$};
\draw (367.6,119.73) node [anchor=north west][inner sep=0.75pt]  [font=\small]  {$1/R$};
\draw (152.6,149.73) node [anchor=north west][inner sep=0.75pt]  [font=\small]  {$1/R$};
\draw (292.85,62.2) node [anchor=north west][inner sep=0.75pt]    {$\gamma $};
\draw (184.6,101.7) node [anchor=north west][inner sep=0.75pt]    {$\sigma $};
\draw (251.1,170.75) node [anchor=north west][inner sep=0.75pt]    {$\varphi _{-}$};
\draw (248.43,94.95) node [anchor=north west][inner sep=0.75pt]    {$\varphi _{+}$};
\end{tikzpicture}

\caption{Curves $\gamma$, $\sigma$, $\varphi_+$ and $\varphi_-$.}
\end{figure}

Note that $\gamma$ is an arc contained in the circumference of center at $0$ and radius $R$, while $\sigma$ is another arc contained in the circumference of center also at $0$ but with radius $1/R$. Both arcs are joined by the horizontal segments $\varphi_+$ and $\varphi_-$, whose distance to $\R$ is $1/R$. Define $\Gamma = \gamma \cup \varphi_- \cup \sigma \cup \varphi_+$ as well.

Given $z \in \H$, consider
$$F_z(w) = \dfrac{1}{(w-z)w^\alpha}, \quad w \in \C \setminus [0,+\infty)$$
where $w^\alpha$ is defined by using the branch of the argument function from $\C \setminus [0,+\infty)$ onto $(0,2\pi)$. Note that $\text{Res}(F_z,z) = 1/z^{\alpha}$. Therefore, given $z \in \H$ and as long as $R > 0$ is big enough, Cauchy's Residue Theorem implies that
$$\int_{\Gamma}F_z(w)dw = \dfrac{2\pi i}{z^{\alpha}}.$$
On the other hand,
$$\abs{\int_{\sigma}F_z(w)dw} \leq \dfrac{2\pi}{(\abs{z}-1/R)R^{1-\alpha}},$$
and so
$$\lim_{R \to + \infty}\int_{\sigma}F_z(w)dw = 0.$$
Similarly,
$$\lim_{R \to + \infty}\int_{\gamma}F_z(w)dw = 0.$$

Moreover, notice that $w^{\alpha} = \exp(\alpha\log(\abs{w}))\exp(i\alpha\arg(w))$, where $\arg \colon \C \setminus [0,+\infty) \to (0,2\pi)$, and then
$$\lim_{R \to + \infty}\int_{\varphi_+}F_z(w)dw = p_{\alpha}(z), \quad \lim_{R \to + \infty}\int_{\varphi_-}F_z(w)dw = -\dfrac{1}{\exp(i2\pi\alpha)}p_{\alpha}(z).$$

Summing up, one deduces that
$$p_{\alpha}(z) = \dfrac{\pi\exp(i\pi\alpha)}{\sin(\pi\alpha)z^{\alpha}}, \quad z \in \H.$$
From this, it is possible to check that $\arg(p_{\alpha}(r\exp(i\theta)) = \theta$ for $\theta = \pi\alpha/(1+\alpha)$. This means that the half-line with principal argument given by $\pi\alpha/(1+\alpha)$ is invariant by $f$, so that Theorem \ref{teo:0HS} assures that
$$\mathrm{Slope}[f_{\alpha},z] = \left\lbrace \dfrac{\pi\alpha}{1+\alpha}\right\rbrace, \quad z \in \H.$$
\end{example}
\begin{example}
Similarly as before, for $0 < \alpha < 1$, let $\mu_{\alpha}$ be the positive finite measure on $\R$ given by
$$\dfrac{d\mu_{\alpha}}{dm}(t) = \dfrac{1}{(1+t^2)\abs{t}^{\alpha}}\chi_{(-\infty,0)}(t), \quad t \in \R.$$
Following the same ideas and notations of the previous example, it is possible to see that
$$\mathrm{Slope}[f_{\alpha},z] = \left\lbrace \pi - \dfrac{\pi\alpha}{1+\alpha} \right\rbrace, \quad z \in \H.$$
\end{example}

\section{Examples with undefinite slope}
The goal of this section is to explicitly construct parabolic functions $f \colon \H \to \H$ of zero hyperbolic step whose Denjoy-Wolff point is $\infty$ such that $\text{Slope}[f,z] = [0,\pi]$ or $\text{Slope}[f,z] = [0,\pi/2]$. To achieve this, consider $\{a_k\}\subset (0,+\infty)$, $\{\gamma_k\} \subset \R\setminus\{0\}$ two sequences satisfying
\begin{equation}
\label{cond:tinl1}
\sum_{k = 1}^{+\infty}\dfrac{a_k}{\abs{\gamma_k}} < +\infty,
\end{equation}
and define the holomorphic function $f \colon \H \to \H$ given by
$$f(z) = z + \sum_{k=1}^{+\infty}p_k(z), \quad p_k(z) = \dfrac{a_k}{\gamma_k-z}, \quad z \in \H.$$
Following the notation on Theorem \ref{teo:expr}, it is possible to check that
$$f(z) = z + \beta + \int_{\R}\dfrac{1+tz}{t-z}d\mu(t), \quad z \in \H,$$
where
$$\mu = \sum_{k = 1}^{+\infty}\dfrac{a_k}{1+\gamma_k^2}\delta_{\gamma_k}, \quad \beta = \int_{\R}td\mu(t), \quad \int_{\R} \abs{t} d\mu(t) < +\infty.$$
Our goal is to show that, under suitable conditions on $\{a_k\}$ and $\{\gamma_k\}$, one has $\text{Slope}[f,z] = [0,\pi]$ or $\text{Slope}[f,z] = [0,\pi/2]$.

Let us also recall that the orbit of every $z_0 \in \H$ is denoted as
$$z_{n+1} = f(z_n), \quad z_n = x_n+iy_n,$$
where $\{y_n\}$ is a non-decreasing sequence since $\mathrm{Im}(f(z)-z) \geq 0$ for all $z \in \H$.

\subsection{$\text{Slope}[f,z] = [0,\pi]$} Suppose that
\begin{align}
z_0 = i, \label{cond:initialpoint} \\
\gamma_k = (-1)^k\abs{\gamma_k}, \quad \abs{\gamma_1} \geq 1, \quad 4\abs{\gamma_k} \leq \abs{\gamma_{k+1}},  \label{cond:gammak} \\
\sum_{l=1}^k a_l \leq \dfrac{\abs{\gamma_k}}{4k}, \quad \sum_{l = k+1}^{+\infty}\dfrac{a_l}{\abs{\gamma_l}} \leq \dfrac{\abs{\gamma_k}}{8k}, \label{cond:small} \\
\sum_{l=1}^{k-1}a_l \leq \dfrac{a_k}{80k}, \quad \sum_{l=k+1}^{+\infty}\dfrac{a_l}{\abs{\gamma_l}} \leq \dfrac{a_k}{160k\abs{\gamma_k}}, \label{cond:control}
\end{align}

For example, these conditions are achieved if $a_k = C_1(k!)^2$ and $\abs{\gamma_k} = C_2(k!)^4$ for big enough $C_1,C_2 > 0$.

In order to clarify the exposition, note that \eqref{cond:gammak} implies that
\begin{equation}
\label{cond:incNk}
\abs{\gamma_k} \uparrow + \infty, \quad \dfrac{\abs{\gamma_k}}{k} \uparrow + \infty, \quad k \to + \infty,
\end{equation}
and \eqref{cond:tinl1} and \eqref{cond:incNk} imply that
\begin{equation}
\label{cond:limLog}
\gamma_k^2/a_k \to +\infty, \quad k \to + \infty.
\end{equation}

\begin{lemma}
\label{lemma:orbitcontrol}
Assume \eqref{cond:tinl1}-\eqref{cond:control}. For any $k \in \N$ and any $n \in \N$ such that $y_n \leq \abs{\gamma_k}/k$, one has $\abs{x_n} \leq 2\abs{\gamma_k}$.
\begin{proof}
Suppose that there exist $k,N \in \N$ with $\abs{x_{N+1}} > 2\abs{\gamma_k}$, $\abs{x_n} \leq 2\abs{\gamma_k}$ for $n \leq N$ and $y_{N+1} \leq \abs{\gamma_k}/k$. Define $\Omega = \{x+iy \in \C : \abs{x}\leq 2\abs{\gamma_k}, \, y \geq 1\}$, and notice that $z_n \in \Omega$ for all $n \leq N$ (see \eqref{cond:initialpoint}). Moreover, if $z =x+iy \in \Omega$ and $l \leq k$, then $\abs{\text{Re}(p_l(z))} \leq  \abs{p_l(z)} \leq a_l$, as $y \geq 1$.

Similarly, if $l > k$, then
$$\abs{\text{Re}(p_l(z))}  \leq  \abs{p_l(z)} \leq \dfrac{a_l}{\abs{\gamma_l-x}} \leq \dfrac{a_l}{\abs{\gamma_l}-2\abs{\gamma_k}} \leq \dfrac{2a_l}{\abs{\gamma_l}},$$
by \eqref{cond:gammak}. Using these inequalities and \eqref{cond:small}, one has
$\abs{\text{Re}(f(z_N)-z_N)} \leq \abs{\gamma_k}/2$, and then $\abs{x_N} \geq 3\abs{\gamma_k}/2 > \abs{\gamma_k}$.

Indeed, if $x_{N+1} < -2\abs{\gamma_k}$ (and similarly if $x_{N+1} > 2\abs{\gamma_k}$), then $-2\abs{\gamma_k} \leq x_N \leq -3\abs{\gamma_k}/2$. Thus, $\text{Re}(p_l(z_N)) > 0$ for all $l \leq k$. Moreover,
$$\text{Re}(p_k(z_N)) = \dfrac{a_k(\gamma_k-x_N)}{(\gamma_k-x_N)^2+y_N^2} \geq \dfrac{a_k}{20\abs{\gamma_k}}.$$
Therefore, by \eqref{cond:control}, $\text{Re}(f(z_N)-z_N) > 0$, and then $x_{N+1} > x_N$. Note this is a contradiction.
\end{proof}
\end{lemma}

\begin{theorem}
Assume \eqref{cond:tinl1}-\eqref{cond:control}. Then, $\mathrm{Slope}[f,z_0] = [0,\pi]$.
\begin{proof}
It suffices to find two subsequences $\{z_{n_k}\}$ and $\{z_{m_k}\}$ with $\arg(z_{n_k}) \to 0$ and $\arg(z_{m_k}) \to \pi$. As the construction is similar for both of them, we will only show how to find the sequence $\{z_{n_k}\}$.

For each $k \in \N$ with $\gamma_k > 0$ (see \eqref{cond:gammak}) there exists $N_k \in \N$ such that $y_{N_k} \leq \gamma_k/k < y_{N_k+1}$. Note that the map $k \mapsto N_k$ may not be injective. 

Consider the set $A = \{k \in \N : x_{N_k} \geq \gamma_k/2 > 0\}$. For every $k \in A$ one has $\arg(z_{N_k}) \leq \arctan(2/k)$. Then, if $A$ is infinite, we can find a sequence with $\arg(z_{n_k}) \to 0$ as $k \to + \infty$.

On the other hand, if $A$ is finite, it is possible to kind $k_0 \in \N$ such that for every $k \in \N$ with $k \geq k_0$ and $\gamma_k > 0$ one has
$$y_{N_k} \leq \gamma_k/k < y_{N_k+1}, \quad -2\gamma_k \leq x_{N_k} \leq \dfrac{\gamma_k}{2},$$
where Lemma \ref{lemma:orbitcontrol} has been used. We claim that there exist $C > 1$ and $k_1 \in \N$, $k_1 \geq k_0$, such that for all $k \in \N$ with $k \geq k_1$ and $\gamma_k > 0$ there exists $M_k \geq N_k$ with $y_{M_k} \leq C\gamma_k/k$ and $x_{M_k} \geq \gamma_k/2$. In particular, $\arg(z_{M_k}) \leq \arctan(2C/k)$. Therefore, the existence of a sequence with $\arg(z_{n_k}) \to 0$ as $k \to + \infty$ follows from the claim.

To prove the claim, fix $k \in \N$ with $k \geq k_0$ and $\gamma_k > 0$, define $\Omega = \{x+iy \in \C : -2\gamma_k \leq x \leq \gamma_k/2, \, 1 \leq y \leq \gamma_k/k\}$, and notice that $z_n \in \Omega$ for $n \leq N_k$. 

If $z = x+iy \in \Omega$ and $l \leq k$, then $\text{Im}(p_l(z)) \leq  \abs{p_l(z)} \leq a_l$. Similarly, if $l > k$, then
$$\text{Im}(p_l(z))  \leq  \abs{p_l(z)} \leq \dfrac{a_l}{\abs{\gamma_l}-\abs{x}} \leq \dfrac{a_l}{\abs{\gamma_l}-2\gamma_k} \leq \dfrac{2a_l}{\abs{\gamma_l}},$$
by \eqref{cond:gammak}. Using these inequalities and \eqref{cond:small}, one has
$\abs{\text{Im}(f(z_{N_k})-z_{N_k})} \leq \gamma_k/k$, and then $\gamma_k/k < y_{N_k+1} \leq 2\gamma_k/k$.

Let us now define $\Omega^* = \{z = x+iy \in \C : -2\gamma_k \leq x \leq \gamma_k/2, \, \gamma_k/k \leq y \leq C\gamma_k/k\}$ for some $C > 1$ to fix. Note that if $z \in \Omega^*$ and $l < k$, then
$$\abs{p_l(z)} \leq \dfrac{a_lk}{\gamma_k}.$$
Similarly, if $l > k$, then
$$\abs{p_l(z)} \leq \dfrac{a_l}{\abs{\gamma_l-x}} \leq \dfrac{a_l}{\abs{\gamma_l}-2\abs{\gamma_k}} \leq \dfrac{2a_l}{\abs{\gamma_l}}$$
by \eqref{cond:gammak}. One can also see that
$$\text{Re}(p_k(z)) \geq \dfrac{a_k\gamma_k/2}{9\gamma_k^2+C^2\gamma_k^2/k^2} \geq \dfrac{a_k}{20\gamma_k},$$
if $k > C$, and
$$\text{Im}(p_k(z)) \leq \dfrac{a_ky}{(\gamma_k-x)^2} \leq \dfrac{4a_ky}{\gamma_k^2}.$$
Therefore, by \eqref{cond:control},
\begin{equation}
\label{eq:boundRe}
\text{Re}(p(z)) \geq \text{Re}(p_k(z)) - \sum_{l\neq k}\abs{p_l(z)} \geq \dfrac{a_k}{40\gamma_k}, \quad z \in \Omega^*.
\end{equation}
Similarly, by \eqref{cond:control},
\begin{equation}
\label{eq:boundIm}
\text{Im}(p(z)) \leq \text{Im}(p_k(z)) + \sum_{l\neq k}\abs{p_l(z)} \leq \dfrac{5a_ky}{\gamma_k^2}, \quad z \in \Omega^*.
\end{equation}

To finish, choose $S_k \in \N$ with
\begin{equation}
\label{eq:S}
\dfrac{100\gamma_k^2}{a_k} \leq S_k \leq \dfrac{\log(C/2)}{\log(1+5a_k/\gamma_k^2)}.
\end{equation}
In order to clarify this, by \eqref{cond:limLog}, define $\alpha_k = \gamma_k^2/a_k$ and note that, if $\log(C/2) > 500$, then
$$\lim_{k \to + \infty} \dfrac{\log(C/2)}{\log(1+5/\alpha_k)}-100\alpha_k = + \infty.$$

Define $L_k = N_k+1+S_k$, and assume that $z_n \in \Omega^*$ for all $N_k +1 \leq n \leq L_k$. In that case, using \eqref{eq:boundIm} and \eqref{eq:S}, one has
$$y_{L_k} \leq \left(1+\dfrac{5a_k}{\gamma_k^2}\right)^{S_k}y_{N_k+1} \leq C\dfrac{\gamma_k}{k},$$
but, using \eqref{eq:boundRe} and \eqref{eq:S},
$$x_{L_k} \geq x_{N_k+1} + S_k\dfrac{a_k}{40\gamma_k} \geq -2\gamma_k + \dfrac{5}{2}\gamma_k = \dfrac{\gamma_k}{2}.$$
From these calculations, one can deduce that there must exist $N_k +1 < M_k < L_k$ with $z_n \in \Omega^*$ for all $N_k + 1 \leq n \leq M_k-1$ and
$$x_{M_k} \geq \dfrac{\gamma_k}{2}, \quad y_{M_k} \leq C\dfrac{\gamma_k}{k}.$$
This is the claim we wanted to prove, with any $C > 1$ such that $\log(C/2) > 500$, and any $k_1 \in \N$ with $k_1 \geq C$. 
\end{proof}
\end{theorem}
\subsection{$\text{Slope}[f,z] = [0,\pi/2]$}
Suppose that

\begin{align}
z_0 = i, \label{cond:initialcond2} \\
\gamma_1 \geq 1, \quad 4\gamma_k \leq \gamma_{k+1}, \label{cond:gamma2} \\
\sum_{l=1}^k a_l \leq \dfrac{\gamma_k}{24k}, \quad \sum_{l = k+1}^{+\infty}\dfrac{a_l}{\gamma_l} \leq \dfrac{\gamma_k}{24k}, \label{cond:12} \\
\sum_{l=1}^{k-1}a_l \leq \dfrac{a_k}{64k^2}, \quad \sum_{l=k+1}^{+\infty}\dfrac{a_l}{\gamma_l} \leq \dfrac{a_k}{100k^2\gamma_k}, \label{cond:22}
\end{align}

For example, these conditions are achieved if $a_k = C_1(k!)^3$ and $\gamma_k = C_2(k!)^6$ for big enough $C_1,C_2 > 0$.

In order to clarify the exposition, note that \eqref{cond:gamma2} implies that
\begin{equation}
\label{cond:incgammakk2}
\abs{\gamma_k} \uparrow + \infty, \quad \dfrac{\abs{\gamma_k}}{k} \uparrow + \infty, \quad k \to +\infty,
\end{equation}
and \eqref{cond:tinl1} and \eqref{cond:incgammakk2} imply that
\begin{equation}
\label{cond:limLog2}
\gamma_k^2/a_k \to +\infty, \quad k \to + \infty.
\end{equation}

\begin{lemma}
\label{lemma:orbitcontrol2}
Assume \eqref{cond:tinl1}, \eqref{cond:initialcond2}-\eqref{cond:22}. For any $k \in \N$ and any $n \in \N$ such that $y_n \leq \gamma_k/k$, one has $0 \leq x_n \leq 3\gamma_k/2$.
\begin{proof}
Suppose that there exists $k,N \in \N$ such that $x_{N+1} \not\in [0,3\gamma_k/2]$, $0 \leq x_n \leq 3\gamma_k/2$ for $n \leq N$ and $y_{N+1} \leq \gamma_k/k$.

Assume that $x_{N+1} > 3\gamma_k/2$, and define $\Omega = \{x+iy \in \C : 0 \leq x \leq 3\gamma_k/2, \, y \geq 1\}$, and notice that $z_n \in \Omega$ for all $n \leq N$. Moreover, if $z =x+iy \in \Omega$ and $l \leq k$, then $\abs{\text{Re}(p_l(z))} \leq  \abs{p_l(z)} \leq a_l$. Similarly, if $l > k$, then
$$\abs{\text{Re}(p_l(z))}  \leq  \abs{p_l(z)} \leq \dfrac{a_l}{\gamma_l-x} \leq \dfrac{a_l}{\gamma_l-2\gamma_k} \leq \dfrac{2a_l}{\gamma_l},$$
by \eqref{cond:gamma2}. Using these inequalities and \eqref{cond:12}, one has
$\abs{\text{Re}(f(z_N)-z_N)} \leq \gamma_k/4$, and then $x_N \in [5\gamma_k/4,3\gamma_k/2]$. Thus, $\text{Re}(p_l(z_N)) < 0$ for all $l \leq k$. Moreover,
$$\abs{\text{Re}(p_k(z_N))} = \dfrac{a_k(x_N-\gamma_k)}{(\gamma_k-x_N)^2+y_N^2} \geq \dfrac{a_k}{5\gamma_k}.$$
Therefore, by \eqref{cond:22}, $\text{Re}(f(z_N)-z_N) < 0$, and then $x_{N+1} < x_N$. Note this is a contradiction.

Similarly, if $x_{N+1} < 0$, notice once more that $z_n \in \Omega$ for all $n \leq N$. Moreover, if $x_N \in [0,\gamma_1)$, then $\text{Re}(p_l(z_N)) > 0$ for all $l \in \N$, and so $x_{N+1} > x_N$, which is a contradiction. In other case, find $L \in \N$ such that $x_N \in [\gamma_L,\gamma_{L+1})$. If $z =x+iy \in \Omega$ and $l \leq L$, then $\abs{\text{Re}(p_l(z))} \leq  \abs{p_l(z)} \leq a_l$. Therefore, by \eqref{cond:12}, one has that
$$\text{Re}(p(z_N)) \geq -\sum_{l = 1}^L\abs{\text{Re}(p(z_N))} \geq -\dfrac{\gamma_L}{2}.$$
Therefore, $x_{N+1}  \geq \gamma_L/2$, which is also a contradiction.
\end{proof}
\end{lemma}

\begin{lemma}
\label{lemma:elevator2}
Assume \eqref{cond:tinl1}, \eqref{cond:initialcond2}-\eqref{cond:22}. Given $k \in \N$, define $\Omega = \{x+iy \in \C: \gamma_k/2 \leq x \leq 3\gamma_k/2, \, 1 \leq y \leq k\gamma_k\}$. Let $N,M \in \N$ be such that $z_n \in \Omega$ for $N \leq n \leq M$ and $z_{M+1} \not\in \Omega$. Then, $\gamma_k/2 \leq x_{M+1} \leq 3\gamma_k/2$ and $y_{M+1} > k\gamma_k$.
\begin{proof}
It is enough to prove that if $z_N \in \Omega$, then $\gamma_k/2 \leq x_{N+1} \leq 3\gamma_k/2$. To do this, split $\Omega = \Omega_1 \cup \Omega_2$, where
$$\Omega_1 = \{x+iy \in \C: 3\gamma_k/4 \leq x \leq 5\gamma_k/4, \, 1 \leq y \leq k\gamma_k\},$$
$$\Omega_2 = \{x+iy \in \C: \gamma_k/4 < \abs{x-\gamma_k} \leq \gamma_k/2, \, 1 \leq y \leq k\gamma_k\}.$$

Assume $z \in \Omega_1$. If $l < k$, then
$$\abs{p_l(z)} \leq \dfrac{2a_l}{\gamma_k},$$
by \eqref{cond:gamma2}. If $l > k$, then
$$\abs{p_l(z)} \leq \dfrac{2a_l}{\gamma_l},$$
also by \eqref{cond:gamma2}. Since $y \geq 1$, $\abs{p_k(z)} \leq a_k$. Therefore, by \eqref{cond:12}, one has
$$\abs{p(z)} \leq \dfrac{\gamma_k}{4},$$
and so
$$\abs{x_N-x_{N+1}} \leq \abs{p(z)} \leq \dfrac{\gamma_k}{4},$$
meaning that
$$\abs{x_{N+1}-\gamma_k} \leq \dfrac{\gamma_k}{2}.$$

Following a similar idea, assume $z \in \Omega_2$. If $l < k$, then
$$\abs{p_l(z)} \leq \dfrac{4a_l}{\gamma_k},$$
by \eqref{cond:gamma2}. If $l > k$, then
$$\abs{p_l(z)} \leq \dfrac{2a_l}{\gamma_l},$$
also by \eqref{cond:gamma2}. Similarly,
$$\abs{\text{Re}(p_k(z))} \geq \dfrac{a_k}{4k^2\gamma_k}.$$
Therefore, by \eqref{cond:22}, the sign of $\text{Re}(p_k(z))$ is the same as the sign of $\text{Re}(p(z))$, that is, the sign of $\gamma_k - x$. This means that
$$\abs{x_{N+1}-\gamma_k} \leq \abs{x_N-\gamma_k} \leq \dfrac{\gamma_k}{2}.$$ 
\end{proof}
\end{lemma}

\begin{theorem}
Assume \eqref{cond:initialcond2}-\eqref{cond:limLog2}. Then, $\mathrm{Slope}[f,z_0] = [0,\pi/2]$.
\begin{proof}
By Lemma \ref{lemma:orbitcontrol2}, notice that $x_n \geq 0$ for all $n \in \N$. Therefore, to show the result, it is enough to find two subsequences $\{z_{n_k}\}$ and $\{z_{m_k}\}$ with $\arg(z_{n_k}) \to 0$ and $\arg(z_{m_k}) \to \pi/2$.

To do this, for each $k \in \N$ define $N_k \in \N$ such that $y_{N_k} \leq \gamma_k/k < y_{N_k+1}$. Note that the map $k \mapsto N_k$ may not be injective. 

Consider the set $A = \{k \in \N : \gamma_k/2 \leq x_{N_k} \leq 3\gamma_k/2\}$. For every $k \in A$ one has $\arg(z_{N_k}) \leq \arctan(2/k)$, and by Lemma \ref{lemma:elevator2} it is also possible to find $M_k \in \N$ with $M_k > N_k$ such that $\arg(z_{M_k}) \geq \arctan(2k/3)$. Then, if $A$ is infinite, the result follows.

On the other hand, if $A$ is finite, Lemma \ref{lemma:orbitcontrol2} assures that it is possible to find $k_0 \in \N$ such that for every $k \in \N$ with $k \geq k_0$ one has
$$y_{N_k} \leq \gamma_k/k < y_{N_k+1}, \quad 0 \leq x_{N_k} \leq \dfrac{\gamma_k}{2}.$$
We claim that there exist $C > 1$ and $k_1 \in \N$, $k_1 \geq k_0$, such that for all $k \in \N$ with $k \geq k_1$ there exist $N'_k,M'_k \in \N$ with $M'_k > N'_k \geq N_k$ such that $\arg(z_{N'_k}) \leq \arctan(2C/k)$ and $\arg(z_{M'_k}) \leq \arctan(2k/3)$. Thus, the results follows from the claim.

To prove the claim, fix $k \in \N$ with $k \geq k_0$, define $\Omega = \{x+iy \in \C : -2\gamma_k \leq x \leq \gamma_k/2, \, 1 \leq y \leq \gamma_k/k\}$, and notice that $z_n \in \Omega$ for $n \leq N_k$.

If $z = x+iy \in \Omega$ and $l \leq k$, then $\text{Im}(p_l(z)) \leq  \abs{p_l(z)} \leq a_l$. Similarly, if $l > k$, then
$$\text{Im}(p_l(z))  \leq  \abs{p_l(z)} \leq \dfrac{a_l}{\gamma_l-x} \leq \dfrac{a_l}{\gamma_l-2\gamma_k} \leq \dfrac{2a_l}{\gamma_l},$$
by \eqref{cond:gamma2}. Using these inequalities and \eqref{cond:12}, one has
$\abs{\text{Im}(f(z_{N_k})-z_{N_k})} \leq \gamma_k/k$, and then $\gamma_k/k < y_{N_k+1} \leq 2\gamma_k/k$.

Let us now define $\Omega^* = \{x+iy \in \C : 0 \leq x \leq \gamma_k/2, \, \gamma_k/k \leq y \leq C\gamma_k/k\}$ for some $C > 1$ to fix. If $z =x+iy \in \Omega^*$ and $l < k$, then
$$\abs{p_l(z)} \leq \dfrac{a_lk}{\gamma_k}.$$
In the case that $l > k$, one has
$$\abs{p_l(z)} \leq \dfrac{a_l}{\gamma_l-\gamma_k/2} \leq \dfrac{2a_l}{\gamma_l},$$
by \eqref{cond:gamma2}. Similarly,
$$\text{Im}(p_k(z)) \leq  \dfrac{4a_ky}{\gamma_k^2}, \quad \text{Re}(p_k(z)) \geq \dfrac{a_k}{4\gamma_k}.$$
Therefore,
\begin{equation}
\label{eq:boundImRe2}
\text{Im}(p(z)) \leq \dfrac{5a_ky}{\gamma_k^2}, \quad \text{Re}(p(z)) \geq \text{Re}(p_k(z))-\sum_{l = 1}^{k-1}\abs{p_l(z)} \geq \dfrac{a_k}{8\gamma_k},
\end{equation}
for all $z \in \Omega^*$, by \eqref{cond:22}.

To finish, choose $S_k \in \N$ with
\begin{equation}
\label{eq:S2}
\dfrac{4\gamma_k^2}{a_k} \leq S_k \leq \dfrac{\log(C/2)}{\log(1+5a_k/\gamma_k^2)}.
\end{equation}
In order to clarify this, by \eqref{cond:limLog2}, define $\alpha_k = \gamma_k^2/a_k$ and note that, if $\log(C/2) > 20$, then
$$\lim_{k \to + \infty} \dfrac{\log(C/2)}{\log(1+5/\alpha_k)}-4\alpha_k = + \infty.$$

Define $L_k = N_k+1+S_k$, and assume that $z_n \in \Omega^*$ for all $N_k + 1 \leq n \leq L_k$. In that case, using \eqref{eq:boundImRe2} and \eqref{eq:S2}, one has
$$y_{L_k} \leq \left(1+\dfrac{5a_k}{\gamma_k^2}\right)^{S_k}y_{N_k+1} \leq C\dfrac{\gamma_k}{k},$$
but, using \eqref{eq:boundImRe2} and \eqref{eq:S2},
$$x_{L_k} \geq x_{N_k+1} + S_k\dfrac{a_k}{8\gamma_k} \geq \dfrac{\gamma_k}{2}.$$
From this calculations, one can deduce that there must exist $N_k +1 < N'_k < L_k$ with $z_n \in \Omega^*$ for all $N_k +1\leq n \leq N'_k-1$ and
$$x_{N'_k}\geq \dfrac{\gamma_k}{2}, \quad y_{N'_k} \leq C\dfrac{\gamma_k}{k}.$$
In particular, $\arg(z_{N'_k}) \leq \arctan(2C/k)$.

Moreover, if $k^2 \geq C$, then $y_{N'_k} \leq k\gamma_k$. Therefore, from Lemma \ref{lemma:elevator2} it is possible to find $M'_k > N'_k$ such that
$$\gamma_k/2 \leq x_{M'_k+1} \leq 3\gamma_k/2, \quad y_{M'_k+1} > k\gamma_k,$$
from which it follows that $\arg(z_{M'_k}) \geq \arctan(2k/3)$.

This is the claim we wanted to prove, with any $C > 1$ such that $\log(C/2) > 20$,
and any $k_1 \in \N$ with $k_1^2 \geq C$.
\end{proof}
\end{theorem}
\begin{remark}
In the literature (see \cite{CDP}), a parabolic function $g \colon \D \to \D$ with Denjoy-Wolff point $\tau$ is said to be of angular-class of order $2$ at $\tau$, denoted as $g \in C^2_A(\tau)$, if
$$g(z) = \tau + (z-\tau) + \dfrac{a_2}{2}(z-\tau)^2 + \gamma(z), \quad z \in \D,$$
where $a_2 \in \C$ and $\gamma$ is a holomorphic function on $\D$ with $\angle\lim_{z \to \tau}\dfrac{\gamma(z)}{(z-\tau)^2}=0$.

In \cite[Proposition 2.1]{CDP}, the authors prove that a parabolic function $g \colon \D \to \D$ is in $C^2_A(\tau)$ if and only if its corresponding conjugated map $f \colon \H \to \H$ whose Denjoy-Wolff point is $\infty$ satisfies that
$$\angle\lim_{z \to \infty}(f(z)-z) \in \C.$$
This can be used to show the key difference between the explicit examples given in this section and the one that Wolff proposed on \cite{WolffSlope}: the former limit is not convergent for Wolff's function, while using Lemma \ref{lemma:anglelimp}, any parabolic function $f \colon \H \to \H$ such that $t \in L^1(\mu)$ satisfies that $\angle\lim_{z \to \infty}(f(z)-z) = \beta \in \R$,
where the notation of Theorem \ref{teo:expr} has been used. Therefore, the examples of this section are endowed with the regularity of the class $C^2_A(\tau)$.

In the continuous setting, the examples developed by Kelgiannis in \cite[Theorem 1]{KelgiannisSlope} do not control how orbits behave. Instead, he used the concept of Koenigs map to propose a geometric description of the slope problem, using several estimates through different harmonic measures. In this sense, the proofs that are proposed in this paper seem rather more straight.
\end{remark}

\end{document}